\documentclass[conference]{IEEEtran}
\IEEEoverridecommandlockouts

\usepackage[T1]{fontenc}
\usepackage{array}
\newcolumntype{P}[1]{>{\centering\arraybackslash}p{#1}}
\newcolumntype{M}[1]{>{\centering\arraybackslash}m{#1}}
\usepackage[latin9]{inputenc}
\usepackage[numbers,compress]{natbib}
\usepackage{amsthm}
\usepackage{amsmath}
\usepackage{algcompatible}
\usepackage{algorithm}
\usepackage{booktabs}
\usepackage{enumerate}
\usepackage{graphicx}
\usepackage{amssymb}
\usepackage{latexsym}
\usepackage{epstopdf}
\usepackage{color}
\usepackage{xcolor}
\usepackage{bbm}
\usepackage{needspace}
\usepackage{color, colortbl}
\usepackage[english]{babel}
\usepackage{tikz}
\usepackage{caption}
\usepackage{subcaption}
\usetikzlibrary{arrows,automata}
\usetikzlibrary{positioning}
\usepackage{filecontents}
\usepackage{mathtools}

\DeclareMathOperator{\Tr}{Tr}

\DeclareCaptionFont{mysize}{\fontsize{8.0}{9.6}\selectfont}
\ifCLASSINFOpdf
\else
\fi

\theoremstyle{plain}
\newtheorem{thm}{\protect\theoremname}
\theoremstyle{plain}

\newtheorem{remark}{Remark}

\newtheorem{define}{Definition}

\newtheorem{assume}{Assumption}

\makeatother

\usepackage{babel}
\providecommand{\lemmaname}{Lemma}
\providecommand{\theoremname}{Theorem}

\newcommand{\prox}{\mathrm{prox}}

\begin{document}

\title{\vspace{0.1in} Robust Convergence Analysis of \\ Three-Operator Splitting}
\author{Han Wang, Mahyar Fazlyab, Shaoru Chen, Victor M. Preciado%

\thanks{Han Wang is with the Department of Applied Mathematics and Computational Science, University of Pennsylvania, Philadelphia, PA, 19104, USA (e-mail: wanghan2@sas.upenn.edu). 

Mahyar Fazlyab, Shaoru Chen and Victor M. Preciado are with the Department of Electrical and Systems Engineering, University of Pennsylvania, Philadelphia, PA, 19104, USA (e-mail: \{mahyarfa, srchen, preciado\}@seas.upenn.edu).}}
	
	\maketitle
	
		\begin{abstract}
Operator splitting methods solve composite optimization
problems by breaking them into smaller sub-problems that can be solved
sequentially or in parallel. In this paper, we propose a unified
framework for certifying both linear and sublinear convergence rates
for three-operator splitting (TOS) method under a variety of
assumptions about the objective function. By viewing the algorithm as
a dynamical system with feedback uncertainty (the oracle model), we
leverage robust control theory to analyze the worst-case
performance of the algorithm using matrix inequalities. We then show
how these matrix inequalities can be used to verify sublinear/linear convergence of the TOS algorithm and guide the search for
selecting the parameters of the algorithm (both symbolically and
numerically) for optimal worst-case performance. We
illustrate our results numerically by solving an input-constrained
optimal control problem.
	\end{abstract}
	
	\section{Introduction}\label{sec:introduction}
	Three-operator splitting methods are aimed to solve optimization problems of the form 
	\begin{equation} \label{eq:ThreeOperatorProb}
	\begin{aligned}
	\min_{x\in\mathbb{R}^d} F(x)=f(x)+g(x)+h(x),
	\end{aligned}
	\end{equation}
	where $f, g$ and $h$ are proper, closed and convex and $h$ is Lipschitz differentiable. Problems of the form \eqref{eq:ThreeOperatorProb} encompass a variety of problems in signal processing, control, and machine learning, such as group LASSO~\cite{jacob2009group}, support vector machines~\cite{o2013splitting}, matrix completion~\cite{candes2010matrix} and optimal control~\cite{stathopoulos2016operator}.
	
	To solve \eqref{eq:ThreeOperatorProb},  \cite{davis2017three} proposed the three-operator splitting (TOS) method outlined below.
	\begin{algorithm} 
		\caption{Three-Operator Splitting (TOS)} \label{alg:TOS}
		\textbf{Input}: $z_0 \in \mathbb{R}^d, \alpha,\lambda >0$.\\
		\textbf{for} $k=0,1,2,\cdots$\\
		\[
		\begin{aligned}
		&x_B^{k}=\prox_{\alpha g}(z^k);\\
		&y^{k}=2{x_B^k-z^k-\alpha \nabla h(x_B^k)};\\
		&x_A^{k}=\prox_{\alpha f}(y^k);\\
		&z^{k+1}=z^k+\lambda(x_A^k-x_B^k);
		\end{aligned}
		\]
		\textbf{endfor}
	\end{algorithm}
	
\noindent In Algorithm \ref{alg:TOS}, $\prox$ is the proximal operator (see Definition~\ref{def:prox}), $\alpha$ is the proximal stepsize and $\lambda$ is the relaxation parameter. \cite{davis2017three} proves that a proper selection of $\lambda$ and $\alpha$ ensures that the sequence $\{x_B^k\}$ converges asymptotically to a minimizer of \eqref{eq:ThreeOperatorProb}. The rate of convergence towards optimality depends on the regularity assumptions about $f, g$ and $h$. In this paper, our goal is to develop a principled and systematic way to analyze the convergence of TOS under various assumptions about $f$, $g$ and $h$.

	\textit{Related Work.} 
To solve problems of the form \eqref{eq:ThreeOperatorProb} with two or more nonsmooth terms, several splitting methods have been proposed. For example, \cite{raguet2013generalized,raguet2015preconditioning} propose a generalized forward-backward splitting algorithm which weakly converges to the minimizer of \eqref{eq:ThreeOperatorProb}.
A primal-dual method based on reformulating~\eqref{eq:ThreeOperatorProb} as a saddle point problem has been proposed by \cite{condat2013primal, vu2013splitting,li2016fast,yan2018new,pedregosa2018adaptive}. \cite{pedregosa2018adaptive,chambolle2016ergodic} and~\cite{davis2017three} prove the $O(1/k)$ ergodic convergence rate on the saddle point suboptimality and function value suboptimality, respectively. When both $f(x)$ and $h(x)$ are Lipschitz differentiable, \cite{davis2017three, pedregosa2018adaptive} give an $O(1/k)$ convergence proof in terms of the objective function value suboptimality. Furthermore, they derive linear convergence under stronger assumptions. 

Recently, there has been a surge of interest in analysis and design of optimization algorithms using robust control and semidefinite programming \cite{lessard2016analysis,fazlyab2018analysis,hu2017dissipativity,van2017fastest,fazlyab2018design,seidman2019control,mohammadi2019performance,hassan2019proximal}. The main idea is to view the worst-case convergence analysis of optimization algorithms as robust stability analysis of a linear dynamical system in feedback connection with an uncertain component \cite{lessard2016analysis}. This perspective is useful in that it allows us to provide either new bounds or design new optimization algorithms in a systematic manner.
    	
\textit{Our Contribution.} The TOS Algorithm can be viewed as a linear dynamical system driven by the nonlinear operators $\prox_{\alpha f}, \prox_{\alpha g}$ and $\nabla h$. For analyzing the convergence of the algorithm to its fixed point(s), we use the framework of quadratic constraints to abstract these nonlinearities using the assumptions made about the oracle models of $f, g$ and $h$. 
	We then define a Lyapunov function for the algorithm whose decrease along the trajectories directly certifies convergence to an optimal solution at a specific rate. We then find sufficient conditions, in terms of matrix inequalities, to guarantee this decrease condition.  Depending on the regularity assumptions, we provide this convergence rate in terms of either the distance to the optimal solution, the norm of the optimality residual, or the objective value. These matrix inequalities can be used to select the parameters for optimal worst-case performance. 
	
	The rest of the paper is organized as follows. In Section~\ref{sec:pre}, we provide preliminaries and background. Then we analyze the sublinear and linear convergence of the algorithm under different sets of assumptions in Section~\ref{sec:sublinear} and Section~\ref{sec:linear}, respectively. In Section~\ref{sec:numerical}, we solve an optimal control problem to illustrate our analysis of convergence and parameter selection. Section~\ref{sec:conclusion} concludes the paper.
	

\section{Preliminaries}\label{sec:pre}
We denote by $I_d$ the $d$-dimensional identity matrix. For a function $f:\mathbb{R}^d \rightarrow \mathbb{R}$, the domain of $f$ is $\text{dom} \ f = \lbrace x \in \mathbb{R}^d \mid f(x) < \infty \rbrace$. The subdifferential of a convex function $f$ at point $x$ is the set $\partial f(x)  = \{ g \in \mathbb{R}^d \mid f(y) - f(x) \geq g^T (y - x), \forall y \in \text{dom} \ f \}$. With abuse of notation, we will denote $\partial f(x)$ as the subgradient of $f$ which is an element of the subdifferential of $f$ at $x$ as well. In this paper, unless explicitly specified otherwise, the norm $|| x ||$ of a vector $x$ denotes the $2$-norm of $x$. We denote the Kronecker product by $\otimes$ and the set of $d \times d$ symmetric matrices by $\mathbb{S}^d$. The spectral norm (maximum singular value) of a matrix $X$ is denoted by $\lVert X \rVert_2$.

    \begin{define} \label{def:prox} \textbf{(Proximal operator)} The proximal mapping of a convex function $f: \mathbb{R}^d \rightarrow \mathbb{R}$ is defined by
    \begin{equation}
    \begin{aligned}
    \prox_{f} (x)=\arg\min_{y} f(y)+\frac{1}{2} \lVert x-y \rVert^2.
    \end{aligned}
    \end{equation}	
    \end{define}
    \begin{define}\textbf{(Lipschitz differentiability)} A function $f:\mathbb{R}^d \rightarrow \mathbb{R}$ is $L_f$-Lipschitz differentiable on $\mathcal{S} \subseteq \text{dom} f$ if 
    \begin{equation}\label{eq:Lsmoothness}
    \begin{aligned}
    \lVert \nabla f(x)-\nabla f(y)\rVert \le L_f \lVert x-y \rVert
    \end{aligned} 
    \end{equation}
    holds for some $L_f>0$ and all $x,y \in \mathcal{S}$. Lipschitz differentiability implies 
    \begin{equation*}
          f(y) \le f(x)+\nabla f(x)^T (y-x) + \frac{L_f}{2} \lVert y-x\rVert^2
    \end{equation*}
    for all $x,y \in \mathcal{S}$.
    \end{define}
    \begin{define}
    \textbf{(Strong convexity)} A function $f:\mathbb{R}^d \rightarrow \mathbb{R}$ is called $m_f$-strongly convex on $\mathcal{S} \subseteq \text{dom} f$ ($m_f \geq 0$)
    if 
    \begin{equation}\label{eq:strong convex}
    \begin{aligned}
     (x-y)^T (p-q) &\ge m_f \lVert x-y \rVert^2
     \end{aligned}
    \end{equation}
holds for all $x,y \in \text{dom} \ f, \forall p \in \partial f(x),  \forall q \in \partial f(y)$.   
\end{define}

We denote the function class satisfying \eqref{eq:Lsmoothness} and \eqref{eq:strong convex} by $\mathcal{F} (m_f, L
_f)$. When $f$ is not differentiable, we have $L_f=\infty$ and we adopt the convention $1/L_f=0$.

\begin{define}(\textbf{Incremental quadratic constraints}~\cite{accikmecse2011observers}) A nonlinear function $\phi: \mathbb{R}^d \rightarrow \mathbb{R}^d$ satisfies the incremental quadratic constraint defined by $Q$ if for all $x, y \in \text{dom} \ \phi$
\begin{equation} \label{eq:QC}
\begin{bmatrix}
x-y\\
\phi (x)-\phi (y)
\end{bmatrix}^T
Q
\begin{bmatrix}
x-y\\
\phi (x)-\phi (y)
\end{bmatrix}
\ge 0,
\end{equation}
where $Q \in \mathbb{S}^{2d}$ is a symmetric, indefinite matrix.
\end{define}
A differentiable function $f$ belongs to the class $\mathcal{F}(m, L)$ on $\mathcal{S}$ if and only if the gradient function $\nabla f$ satisfies the incremental quadratic constraint in \eqref{eq:QC} where $Q = Q(m, L)$ is given by~\cite{nesterov1998introductory, lessard2016analysis} 
\begin{equation*} \label{eq:Q_phi}
    Q(m, L)=
    \begin{bmatrix}
    -\frac{m L}{m + L}&1/2\\
    1/2&-\frac{1}{m + L}
    \end{bmatrix} \otimes I_d.
\end{equation*}
	\subsection{Convergence Analysis of Three-Operator Splitting} \label{sec:analysisScheme}
	
	The TOS algorithm can be equivalently written in terms of the subgradients of $f$ and $g$ as:
	\begin{equation} \label{eq:UpdatingRule}
	\begin{aligned} 
	x_B^{k} &= z^k - \alpha \partial g(x_B^k)\\
	y^k     &= 2x_B^k - z^k -\alpha \nabla h(x_B^k)\\
	x_A^{k} &= y^k - \alpha \partial f(x_A^k)\\
	z^{k+1} &= z^k + \lambda (x_A^k - x_B^k).
	\end{aligned}
	\end{equation}
	The fixed points of the above iterations satisfy the following equations:
	\begin{equation} \label{eq:FixedPoint}
	\begin{aligned} 
	x_B^{\star} &= z^{\star} - \alpha \partial g(x_B^{\star})\\
	y^{\star} &= 2x_B^{\star} - z^{\star} -\alpha \nabla h(x_B^{\star})\\
	x_A^{\star} &= y^{\star} - \alpha \partial f(x_A^{\star})\\
	x_A^{\star} &= x_B^{\star}.
	\end{aligned}
	\end{equation}
	 By adding up both sides of~\eqref{eq:FixedPoint}, we find that the fixed points of the TOS algorithm satisfy
	\begin{equation}\label{eq:FirstOrderCond}
	\partial f(x_A^{\star})+\partial g(x_B^{\star})+\nabla h(x_B^{\star})= 0,
	\end{equation}
	which is the first-order optimality condition for problem \eqref{eq:ThreeOperatorProb}. 
	
	By defining the variable $u^k = x_A^k - x_B^k$, the iterates of the TOS algorithm can be viewed as a linear system of the form 
	\begin{equation*} \label{eq:LinearSys}
	z^{k+1} = z^k + \lambda u^k    
	\end{equation*}
	with state $z^k \in \mathbb{R}^d$, control input $u^k \in \mathbb{R}^d$ and state feedback control law 
	\begin{equation*}
	\begin{aligned}
	u^k = & \quad \psi(z^k) \\
	    = & \quad \prox_{\alpha f}( 2\prox_{\alpha g} (z^k) - z^k - \alpha \nabla h(\prox_{\alpha g} (z^k)) ) \\
	    & \quad - \prox_{\alpha g} (z^k).
	\end{aligned}
	\end{equation*}
	In Section~\ref{sec:sublinear} and Section~\ref{sec:linear}, we analyze the sublinear and linear convergence of the TOS algorithm using Lyapunov arguments.
	
	
	
	\ifx
	Throughout the paper, we will use the following energy functional:
	\begin{align}
	    V_k &= \|z^k-z^{\star}\|^2 + \theta_1 \sum_{i=0}^{k-1}(F(x_B^i)-F(x_B^\star)) \\ & \qquad + \theta_2 \sum_{i=0}^{k-1}\lVert \partial f(x_A^i) + \partial g(x_B^i) + \nabla h(x_B^i) \rVert^2 \notag
	\end{align}
	where $\theta_1,\theta_2 \geq 0$.\fi


	\section{Sublinear Convergence of TOS} \label{sec:sublinear}
	\subsection{Case 1: One Lipschitz Operator}\label{sec:Case1}
	In this part, we will investigate the convergence rate of TOS algorithm when $f, g$ and $h$ are proper, closed and convex and $h$ is Lipschitz differentiable.
	%
	We use the Lyapunov function 
	\begin{equation*} \label{eq:Case1LyapVk}
	V_k=\lVert z^k-z^{\star} \rVert^2 + \theta \sum_{i=0}^{k-1} \lVert \partial f(x_A^i) + \partial g(x_B^i) + \nabla h(x_B^i) \rVert^2
	\end{equation*}
	where $\theta>0$. Using this definition, we can show that the condition $V_{k+1} \leq V_k$ implies  
	\begin{equation} \label{eq:Lemma1Lyap}
		\underset{i = 0, \cdots, k-1}{\min} \lVert \partial f(x_A^i) + \partial g(x_B^i) + \nabla h(x_B^i) \rVert^2 \leq \frac{\lVert z^0 - z^{\star} \rVert^2}{\theta k}.
		\end{equation}
		
	
	
%
In the following theorem, we derive a matrix inequality in terms of $\alpha,\lambda$ and $\theta$ as a sufficient condition to guarantee that $V_{k+1} \leq V_{k}$ for all $k$. 
	
	\begin{thm} \label{Thm2} Let $m_f=m_g=m_h=0$ and $L_h<L_f=L_g=\infty$.
		Define $W_0$, $Q_1$, $Q_2$ and $Q_3$ as follows:
		\begin{subequations} \label{eq:Thm2Matrix}
			\begin{align}
			&W_0=
			\begin{bmatrix}
			\lambda^2+\theta/\alpha^2&0&-\lambda^2-\theta/\alpha^2&-\lambda\\
			0 &0&0&0\\
			-\lambda^2-\theta/\alpha^2&0&\lambda^2+\theta/\alpha^2&\lambda\\
			-\lambda&0&\lambda&0\\
			\end{bmatrix}\otimes I_d, \\
			&Q_1=
			\begin{bmatrix}
			\alpha I_d & -I_d\\
			0&0\\
			0&0\\
			0& I_d
			\end{bmatrix} 
			Q(m_g,L_g)
			\begin{bmatrix}
			\alpha I_d &0&0&0\\
			-I_d &0&0& I_d
			\end{bmatrix},\\
			&Q_2= 
			\begin{bmatrix}
			\alpha I_d&2I_d\\
			0&-I_d\\
			0&0\\
			0&-I_d
			\end{bmatrix}
			Q(m_h,L_h)
			\begin{bmatrix}
			\alpha I_d&0&0&0\\
			2 I_d&-I_d&0&-I_d
			\end{bmatrix},\\
			&Q_3=
			\begin{bmatrix}
			0&0\\
			0&I_d\\
			\alpha I_d&-I_d\\
			0&0
			\end{bmatrix}
			Q(m_f,L_f)
			\begin{bmatrix}
			0&0&\alpha I_d&0\\
			0&I_d&-I_d&0
			\end{bmatrix}.
			\end{align}
		\end{subequations}
		Suppose there exist $\lambda,\alpha,\theta>0,\sigma_1,\sigma_2,\sigma_3\geq 0$ such that the following matrix inequality
		\begin{equation} \label{eq:Case1LMI}
		W_0 + \sigma_1 Q_1 + \sigma_2 Q_2 + \sigma_3 Q_3 \preceq 0
		\end{equation}
		holds, then for all $f, g \in \mathcal{F}(0, \infty), h \in \mathcal{F}(0, L_h)$,
		Algorithm~\ref{alg:TOS} satisfies
		\begin{equation} \label{eq:Thm2Lyap}
		\underset{i = 0, \cdots, k-1}{\min} \lVert \partial f(x_A^i) + \partial g(x_B^i) + \nabla h(x_B^i) \rVert^2 \leq \frac{\lVert z^0 - z^{\star} \rVert^2}{\theta k}.
		\end{equation}
	\end{thm}
	
	\begin{proof}
    See Appendix \ref{pro:thm2}.
	\end{proof}
	By Theorem~\ref{Thm2}, any $(\lambda,\alpha,\theta,\sigma_1,\sigma_2,\sigma_3)$ that satisfy the matrix inequality \eqref{eq:Case1LMI} certifies an $\mathcal{O}(1/k)$ convergence of the TOS algorithm.
	%
	We can show that the matrix inequality~\eqref{eq:Case1LMI} has a symbolic solution:
	\begin{equation*}
	\begin{aligned}
	    &\alpha=(2-\lambda)/L_h, \quad \sigma_1=\sigma_2 = \sigma_3 =\frac{2 \lambda}{\alpha},\\
	    & \theta=(2-\lambda)^3 \lambda/(2 L_h^2),
	    \end{aligned}
	\end{equation*}
	for $\lambda \in (0,2)$. The solution is found by applying Sylvester's criterion~\cite{horn2012matrix} in Wolfram Mathematica. 
	
	
	
	\begin{remark} \label{rem: sublinear best lambda}
		To obtain the best convergence rate, we need to make $\theta$ as large as possible in \eqref{eq:Thm2Lyap}. Since $\theta=(2-\lambda)^3 \lambda/(2 L_h^2)$, a straightforward calculation shows that $\theta$ obtains the maximal value if we set $\lambda=\frac{1}{2}$. Then the following convergence hold:
		\begin{equation*}
		\begin{aligned}
		\underset{i = 0,\cdots,k-1}{\min} \lVert \partial f(x_A^i) + \partial g(x_B^i) & + \nabla h(x_B^i) \rVert^2 \\
		& \leq \frac{32 L_h^2 \lVert z^0 - z^{\star} \rVert^2}{27 k},
		\end{aligned}
		\end{equation*}
		or equivalently
		\begin{equation*}
		\underset{i = 0, \cdots, k-1}{\min} \lVert x_B^i - x_A^i \rVert^2 \leq \frac{8 \lVert z^0 - z^{\star} \rVert^2}{3 k}.
		\end{equation*}
	\end{remark}
	
	Next, we will prove the sublinear convergence of the TOS algorithm when both $f$ and $h$ are Lipschitz differentiable.
	
	
	\subsection{Case 2: Two Lipschitz Operators} \label{section:Case2}
	In this part, we assume that $g \in \mathcal{F}(0, \infty), f \in \mathcal{F}(0, L_f), h \in \mathcal{F}(0, L_h)$ with $L_f, L_h < \infty$. 
	We define the Lyapunov function
	\begin{equation*} \label{eq:Case2Lyap}
	V_k =\lVert z^k-z^{\star}\rVert^2 + \theta \sum_{i=0}^{k-1}[F(x_B^i)-F(x_B^{\star})].
	\end{equation*}
	When the Lyapunov function decreases along the trajectories of TOS, we can guarantee an $\mathcal{O}(1/\theta k)$ convergence rate in terms of objective values:
    \begin{equation*} \label{eq:Case2ConvergenceRate}
    \underset{i = 0,\cdots, k-1}{\min}[F(x_B^i)-F(x_B^{\star})] \le \frac{1}{\theta k} \lVert z^{0}-z^{\star} \rVert^2.
    \end{equation*}
    In the next theorem, we derive a matrix inequality that ensures $V_{k+1} \leq V_k$ for all $k$.
	
	

	
	\begin{thm} \label{thm:Case2}
		Define $W_1$ to be
		\begin{equation*}
		W_1=
		\left[
		\begin{array}{c|c}
		A& B \\ \hline 
		C& D
		\end{array}
		\right] \otimes I_d,
		\end{equation*}
		where
		\begin{equation*}
		\begin{aligned}
		&A=\begin{bmatrix}
		\lambda^2+(\frac{1}{\alpha}+\frac{L_f}{2}-\frac{2}{\alpha^2 L_h} )\theta&\frac{\theta}{\alpha^2 L_h}\\
		\frac{\theta}{\alpha^2 L_h}&-\frac{\theta}{2 \alpha^2 L_h}
		\end{bmatrix},\\
		&B=\begin{bmatrix}
		-\lambda^2-\theta(\frac{1}{2 \alpha}+\frac{L_f}{2})&-\lambda+\frac{\theta}{\alpha^2 L_h}\\
		0&-\frac{\theta}{2\alpha^2 L_h}
		\end{bmatrix},\\
		&C=\begin{bmatrix}
		-\lambda^2-\theta(\frac{1}{2 \alpha}+\frac{Lf}{2})&0\\
		-\lambda+\frac{\theta}{\alpha^2 L_h}&-\frac{\theta}{2\alpha^2 L_h},
		\end{bmatrix},\\
		&D=\begin{bmatrix}
		\lambda^2+\frac{\theta L_f}{2}&\lambda\\
		\lambda&-\frac{\theta}{2 \alpha^2 L_h}
		\end{bmatrix}.
		\end{aligned}
		\end{equation*}	
		Let $Q_1, Q_2, Q_3$ have the same form as in \eqref{eq:Thm2Matrix} with $m_f = m_g = m_h = 0$ and $L_f, L_h < L_g = \infty$. If there exist parameters $\theta, \alpha, \lambda>0$ and $\sigma_1,\sigma_2,\sigma_3 \geq 0$ such that the following matrix inequality
		\begin{equation} \label{eq:Case2LMI}
		W_1 + \sigma_1 Q_1 + \sigma_2 Q_2 + \sigma_3 Q_3 \preceq 0
		\end{equation}
		holds, then for all $g \in \mathcal{F}(0, \infty), f \in \mathcal{F}(0, L_f), h \in \mathcal{F}(0, L_h)$, Algorithm~\ref{alg:TOS} satisfies
		\begin{equation} \label{eq:Case2Convergence}
		\underset{i = 0,\cdots, k-1}{\min}[F(x_B^i)-F(x_B^{\star})] \le \frac{1}{\theta k} \lVert z^{0}-z^{\star} \rVert^2.
		\end{equation}
	\end{thm}
	
	\begin{proof}
    See Appendix~\ref{pro:Thm5}.
	\end{proof}
	For a given stepsize $\alpha > 0 $ and relaxation parameter $\lambda>0$ the best worst-case convergence rate corresponds to maximizing $\theta$ subject to the LMI in \eqref{eq:Case2LMI}, which is an SDP. Note that we can use Schur Complements to convexify \eqref{eq:Case2LMI} with respect to $\lambda$, as follows.
	%
	%
	%
	%
	First, define
	\begin{equation*}\label{eq:Schur}
	\begin{aligned}
	&\eta =
	\begin{bmatrix}
	\lambda&0&-\lambda&0
	\end{bmatrix}^T \otimes I_d,\\
	&M_2 = W_1-\eta \eta^T+\sigma_1 Q_1+\sigma_2 Q_2 +\sigma_3 Q_3,
	\end{aligned}
	\end{equation*}
	and 
	\begin{equation*}
	\begin{aligned}
	\widetilde{W_1}=
	\begin{bmatrix}
	M_2&\eta\\
	\eta^T& -I_d
	\end{bmatrix}.
	\end{aligned}
	\end{equation*}
	Then \eqref{eq:Case2LMI} reads as
	\begin{equation*} 
	M_2+\eta  \eta^T \preceq 0.
	\end{equation*}
    Since $M_2 +\eta \eta^T$ is the Schur complement of $\widetilde{W_1}$, 
	\eqref{eq:Case2LMI} is equivalent to $\widetilde{W_1} \preceq 0$, which is linear in $\lambda$.
	As a result, finding the best convergence rate is equivalent to solving the following SDP:
	\begin{equation*} \label{eq:thetaSDP}
	\begin{aligned}
	\underset{\theta,\lambda,\sigma_1,\sigma_2,\sigma_3}{\text{maximize}} & \quad \theta\\
	\text{subject to}& \quad \widetilde{W_1} \preceq 0\\
	&\quad \theta, \lambda>0 ,\sigma_1,\sigma_2,\sigma_3 \geq 0.
	\end{aligned}
	\end{equation*}
	%
	%
	Finally, we can solve the SDP over a range of stepsizes $\alpha$ to find the best stepsize. We plot $\theta^\star$ over a range of $L_f$ and $L_h$ in Fig.~\ref{fig:theta}. 
	
	\begin{figure}[h]
		\centering
		\includegraphics[width= 0.5\textwidth]{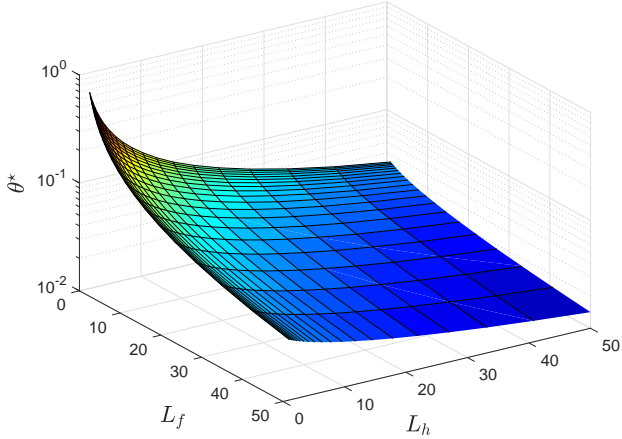} 
		\caption{\label{fig:theta} Optimal sublinear convergence rate $1/(\theta^\star k)$ can be achieved by searching over $\alpha$ for given $f\in \mathcal{F}(0,L_f), h \in \mathcal{F}(0,L_h)$ and $g \in \mathcal{F}(0,\infty)$.} 
	\end{figure}
	
	

In the next section, we analyze the convergence of TOS under strong convexity.	


	
\section{Linear Convergence of TOS} \label{sec:linear}
The TOS algorithm achieves linear convergence rate if there exists a $\rho \in (0,1)$ such that $\lVert z^{k}-z^{\star} \rVert \le \mathcal{O}(\rho^{k})$ for all $k$. 
%
In \cite{davis2017three}, it has been proved that the TOS algorithm achieves linear convergence rate under the following assumption.
\begin{assume} \label{linearassume}
Functions $f, g$ and $h$ in \eqref{eq:ThreeOperatorProb} satisfy $f \in \mathcal{F}(m_f, L_f), g \in \mathcal{F}(m_g, L_g), h \in \mathcal{F}(m_h, L_h)$, respectively and $(m_f+m_g+m_h)(1/L_f+1/L_g)1/L_h>0$ \cite{davis2017three}.
\end{assume}
A closed-form representation of an upper bound on the convergence rate is given in \cite{davis2017three}. However, the form of this bound is complicated and not tight. In \cite{ryu2018operator} the authors improved the upper bound on $\rho$ by formulating an SDP. In contrast, we use Lyapunov functions and incremental quadratic constraints to formulate an SDP that bounds $\rho$ and compare the results with those of \cite{ryu2018operator} 

To begin, we use the  following quadratic Lyapunov function:
\begin{equation*}\label{eq:LyapV_k}
    V_k=\lVert z^k-z^{\star}\rVert^2.
\end{equation*}
If there exists a $\rho \in (0,1)$ such that $V_{k+1} \le \rho^2 V_k$ holds for all $k>0$, then the algorithm is exponentially convergent.




The following theorem provides a sufficient condition in terms of a matrix inequality to achieve linear convergence of the TOS algorithm.

\begin{thm} \label{thm:LinearConvergence}
    Define $W_2$ as 
   \begin{equation*}
         W_2 = 
        \begin{bmatrix}
        \lambda^2&0&-\lambda^2&-\lambda\\
        0&0&0&0\\
        -\lambda^2&0&\lambda^2&\lambda\\
        -\lambda&0&\lambda&1-\rho^2
        \end{bmatrix} \otimes I_d.
    \end{equation*}
If there exist $\sigma_1,\sigma_2,\sigma_3 \geq 0,\alpha, \lambda > 0$ and $\rho \in (0,1)$ such that the following matrix inequality
\begin{equation}\label{eq:LinearConvergenceMI}
    W_2+\sigma_1 Q_1+\sigma_2 Q_2+\sigma_3 Q_3 \preceq 0,
\end{equation}
holds where $Q_1,Q_2,Q_3$ are given in \eqref{eq:Thm2Matrix} with $m_f$, $L_f$, $m_g$, $L_g$, $m_h$, $L_h$ satisfying Assumption~\ref{linearassume}. Then Algorithm~\ref{alg:TOS} satisfies the following linear convergence rate,
\begin{equation} \label{eq:LinearConvergence}
     \lVert z^k-z^{\star} \rVert^2 \le \rho^{2k} \lVert z^0-z^\star \rVert^2.
\end{equation}
\end{thm}

\begin{proof}
See Appendix~\ref{pro:thm6}.
\end{proof}
%
Note that the matrix inequality in \eqref{eq:LinearConvergenceMI} is linear in all the parameters except for $\alpha$ and $\lambda$.
 %
We can use the same technique as shown in Section~\ref{section:Case2} to transform \eqref{eq:LinearConvergenceMI} into an LMI when the stepsize $\alpha$ is fixed. 
Let
	\begin{equation*}\label{eq:Schur_1}
	\begin{aligned}
	& \eta =
	\begin{bmatrix}
	\lambda & 0 & -\lambda&0
	\end{bmatrix}^T \otimes I_d,\\
	& M_3 = W_2-\eta \eta^T+\sigma_1 Q_1+\sigma_2 Q_2 +\sigma_3 Q_3,
	\end{aligned}
	\end{equation*}
	and 
	\begin{equation*}
	\begin{aligned}
	\widetilde{W_2}=
	\begin{bmatrix}
	M_3&\eta\\
	\eta^T&-I_d
	\end{bmatrix}.
	\end{aligned}
	\end{equation*}
	Then by Schur complement,~\eqref{eq:LinearConvergenceMI} is satisfied if and only if
	\begin{equation*} 
	\widetilde{W_2} \preceq 0.
	\end{equation*}
	Therefore, for a given stepsize $\alpha$, the best convergence rate can be found by solving the following SDP:
	\begin{equation} \label{eq:Case2SDP}
	\begin{aligned}
	\underset{\rho^2, \sigma_1, \sigma_2, \sigma_3, \lambda}{\text{minimize}} & \quad \rho^2 \\
	\text{subject to} & \quad \widetilde{W_2} \preceq 0\\
	& \quad \lambda>0 ,\sigma_1,\sigma_2,\sigma_3 \geq 0.
	\end{aligned}
	\end{equation}
	Denote the optimal solution to \eqref{eq:Case2SDP} by $\rho^\star(\alpha)$. Then  by a grid search of $\alpha > 0$, we can find the optimal bound $\rho^\star$ and the optimal stepsize $\alpha^\star$ through
	\begin{equation*}
	\rho^{\star}=\min_{\alpha>0} \rho^{\star} (\alpha)  \quad \text{and} \quad \alpha^{\star}=\arg\min_{\alpha>0} \rho^{\star} (\alpha).
	\end{equation*}
	
	In Fig.~\ref{fig:red_black}, we plot $\alpha \mapsto \rho^\star(\alpha)^2$ and contrast it with the bounds of \cite{ryu2018operator} for various regularity assumptions on $F$. We see from this figure that numerically we achieve the same bounds as in~\cite{ryu2018operator}. In fact, as shown in Appendix~\ref{append:duality},  the formulation in~\eqref{eq:Case2SDP} is the dual of the SDP developed in \cite{ryu2018operator}. 
	
\begin{figure*}[htb] 
\centering 
\begin{subfigure}{0.3\textwidth}
  \includegraphics[width=\linewidth]{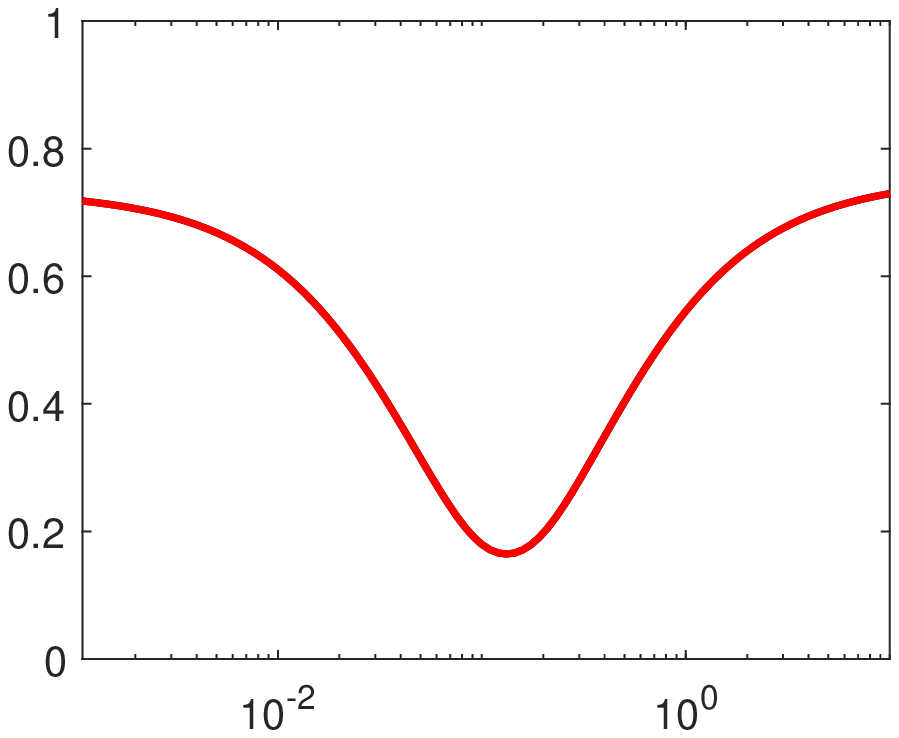}
  \caption{$m_f = 1, Lf = 100/7$ \\ $m_g = 4, L_g = 50, L_h = 1/9$}
  \label{fig:1}
\end{subfigure}\hfil 
\begin{subfigure}{0.3\textwidth}
  \includegraphics[width= \linewidth]{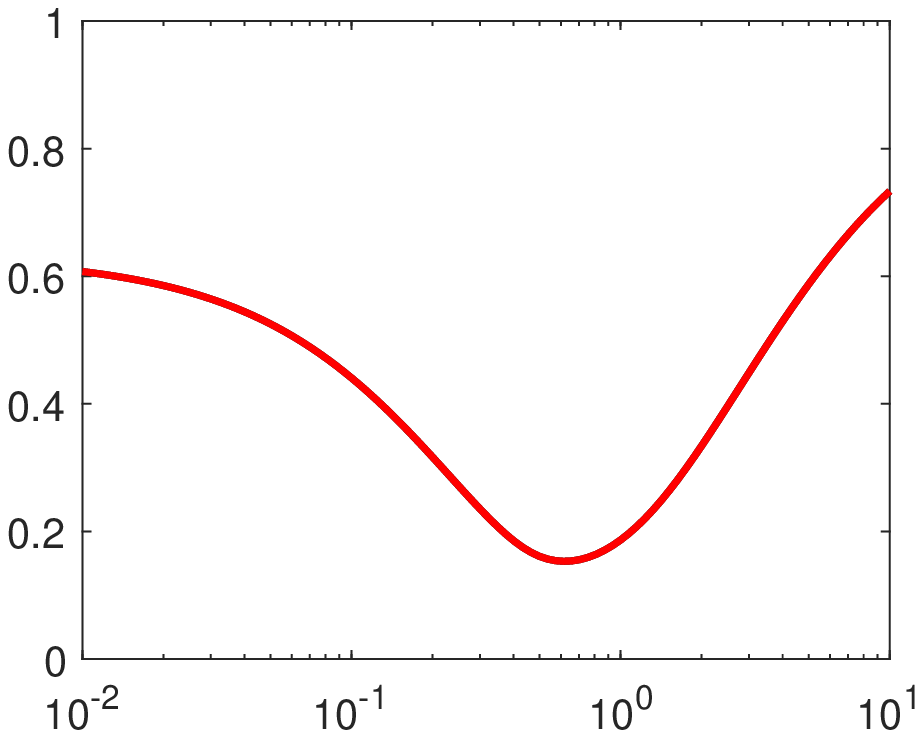}
  \caption{$m_f = 1, Lf = 7, m_g = 0.03$ \\ $ L_g = 2, m_h = 0.01, L_h = 0.05$}
  \label{fig:2}
\end{subfigure}\hfil 
\begin{subfigure}{0.3\textwidth}
  \includegraphics[width=\linewidth]{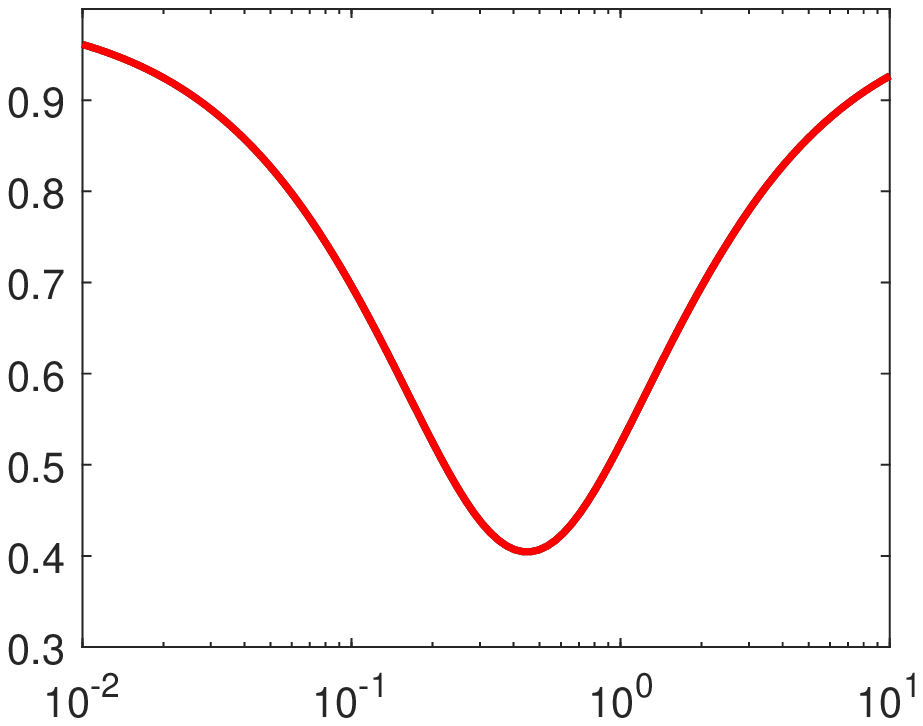}
  \caption{$m_f = 1, L_g = 5, L_h = 1/9$}
  \label{fig:3} \hfil
\end{subfigure}

\medskip
\begin{subfigure}{0.3\textwidth}
  \includegraphics[width=\linewidth]{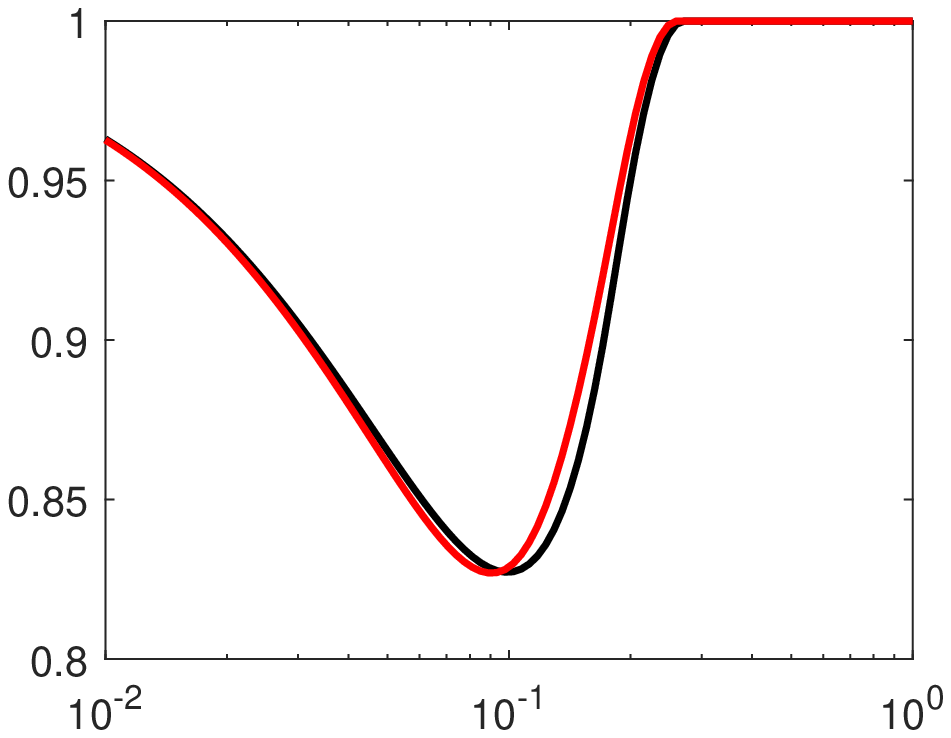}
  \caption{$m_g = 1, L_g = 10, L_h = 20$}
  \label{fig:4}
\end{subfigure}\hfil 
\begin{subfigure}{0.3\textwidth}
  \includegraphics[width=\linewidth]{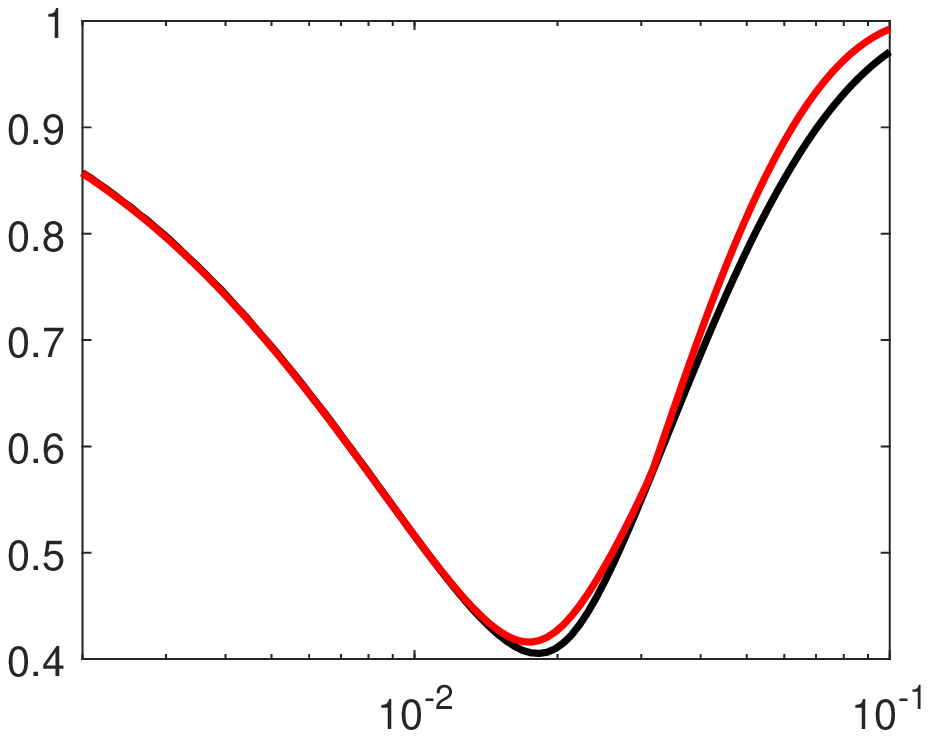}
  \caption{$m_f = 20, L_f = 20, L_h = 70$}
  \label{fig:5}
\end{subfigure}\hfil 
\begin{subfigure}{0.3\textwidth}
  \includegraphics[width=\linewidth]{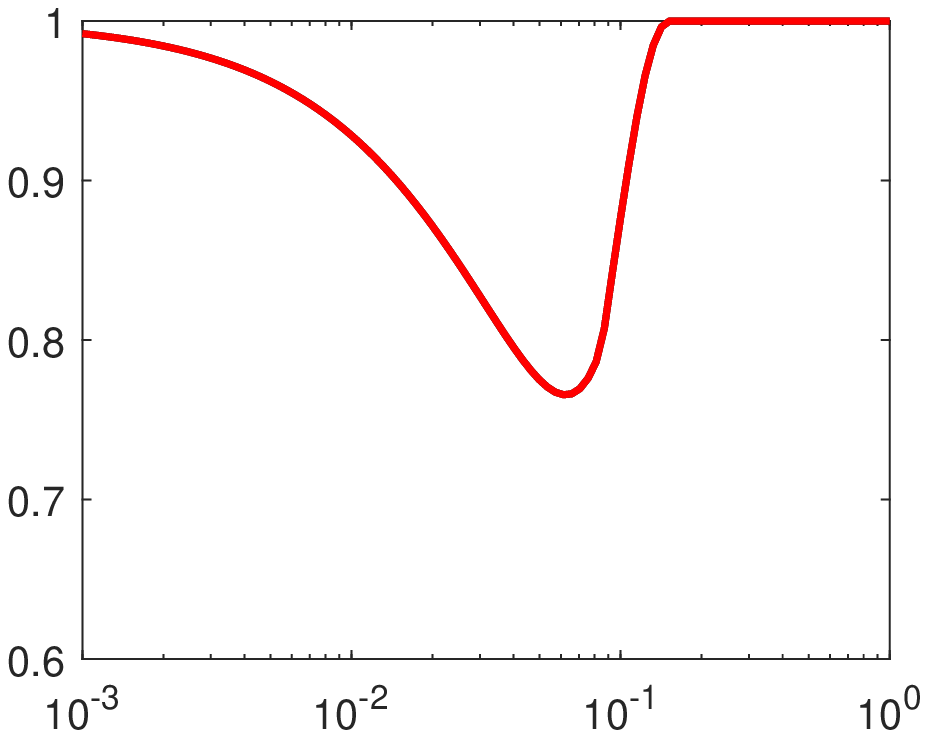}
  \caption{$L_f = 50, m_h = 2, L_h = 30$} \label{fig:where_they_fail}
  \label{fig:6} 
\end{subfigure}
\caption{Plots of $\rho^\star(\alpha)^2$ under different assumptions of $F(x)$. The $x$-axis denotes the stepsize $\alpha$ and the $y$-axis denotes the value of $\rho^\star(\alpha)^2$. Our results are given by red curves while the results in \cite{ryu2018operator} are represented by black curves. When the two curves overlap each other, only the red one is shown.}
\label{fig:red_black}
\end{figure*}

\section{Numerical Example} \label{sec:numerical}
In this section, we validate the parameter selection procedure in Section~\ref{sec:Case1} with a box constrained quadratic optimal control problem from \cite[Sec. \uppercase\expandafter{\romannumeral4}. A]{o2013splitting}:
\begin{equation}\label{eq:sparseLQR}
\begin{aligned}
\underset{x_t\in \mathbb{R}^n, u_t\in \mathbb{R}^m }{\text{minimize}} & \quad \frac{1}{2} \big ( \sum_{t=0}^{N} x_t^T Q_t x_t + \sum_{t=0}^{N-1}u_t^T R_t u_t \big ) \\
\text{subject to} & \quad x_{t+1} = A_t x_t +B_t u_t, \quad t=0,\cdots,N-1\\
                  & \quad \lVert u_t \rVert_{\infty} \le 1, \quad t=0,\cdots,N-1\\ 
                  & \quad x_0 = x_{\text{init}}
\end{aligned}
\end{equation}
where $Q_t \succeq 0$ and $R_t \succ 0$. We use $x = [x_0^T \ \cdots \ x_N^T]^T \in \mathbb{R}^{(N+1)n}$ and $u = [u_0^T \ \cdots \ u_{N-1}^T]^T \in \mathbb{R}^{Nm}$ to denote the concatenated states and control inputs, and $w = [x^T  \ u^T]^T \in \mathbb{R}^{(N+1)n+Nm}$ to denote the state-control trajectory. 

Define the set of state-control pairs that satisfy the dynamics of \eqref{eq:sparseLQR} as 
\begin{equation*}
\mathcal{D}=\{w \mid x_0=x_{\text{init}}, x_{t+1}=A_t x_t + B_t u_t, t = 0, \cdots, N-1 \},
\end{equation*}
and the set of state-control constraints as 
\begin{equation*}
\mathcal{C}=\{w \mid \lVert u \rVert_{\infty} \le 1 \}.
\end{equation*}
The indicator function $I_{\mathcal{D}}$ is defined by 
\begin{equation*}
    I_{\mathcal{D}}(w) = \begin{cases} 0 & w \in \mathcal{D} \\ \infty & \text{otherwise}  \end{cases}
\end{equation*}
and $I_{\mathcal{C}}$ is defined similarly. Then the box constrained optimal control problem \eqref{eq:sparseLQR} can be expressed as 
\begin{equation}\label{eq:LQR}
   \underset{w \in \mathbb{R}^{(N+1)n + Nm}}{\text{minimize}} \quad I_{\mathcal{C}}(w) + I_{\mathcal{D}}(w) + \frac{1}{2} w^T E w
\end{equation}
where 
\begin{equation*}
E = \text{diag}(Q_0, \cdots, Q_N, R_0, \cdots, R_{N-1}).
\end{equation*}
Let $f(w) = I_\mathcal{C}(w), g(w) = I_\mathcal{D}(w)$ and $h(w) = \frac{1}{2} w^T E w$. It can be easily checked that $f, g$ and $h$ are proper, closed and convex and $h$ is Lipschitz differentiable. Then~\eqref{eq:LQR} can be viewed as a three-operator splitting problem and falls into the one Lipschitz operator category in Section~\ref{sec:Case1}.

We consider a medium-size optimal control problem for illustration. For simplicity, we apply a linear time-invariant system with $x_t \in \mathbb{R}^{20}, u \in \mathbb{R}^5, A_t = A, B_t = B$ and constant $Q_t = Q, R_t = R$. The horizon length is $N = 20$. The data are all generated randomly and the matrix $A$ is scaled to be marginally stable, i.e., the largest magnitude of the eigenvalue of $A$ is one. 

According to Remark \ref{rem: sublinear best lambda}, $\lambda = \frac{1}{2}$ gives the fastest worst-case convergence. We solve the problem \eqref{eq:LQR} using the TOS algorithm with different values of $\lambda$ and stepsizes $\alpha=(2-\lambda)/L_h$, where $L_h$ equals to the spectral norm of matrix $E$ in this example. Fig.~\ref{fig:ConvergenceLambda} shows that all convergence rates are dominated by $1/k$ and $\lambda = \frac{1}{2}$ yields the fastest convergence as expected.
\begin{figure}
    \centering
    \includegraphics[width = 0.5 \textwidth]{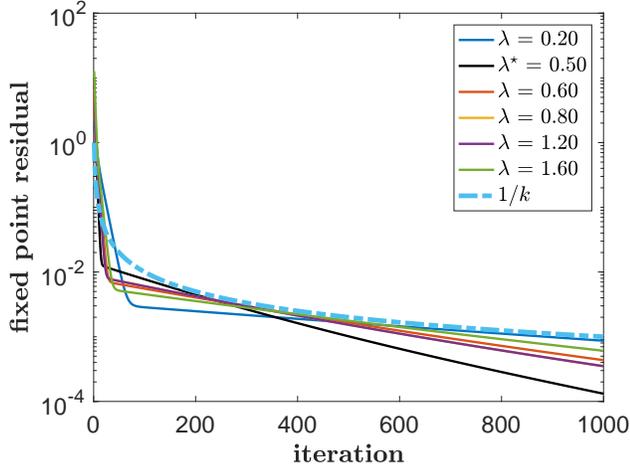}
    \caption{Convergence of TOS on the constrained LQR problem~\eqref{eq:sparseLQR} with varying $\lambda$ and  stepsize $\alpha = (2-\lambda)/L_h$. }\label{fig:ConvergenceLambda}
\end{figure}

\section{Conclusion} \label{sec:conclusion}
In this paper, we  proposed a unified framework, based on Lyapunov functions and quadratic constraints, for convergence rate analysis and parameter selection of the three-operator splitting algorithm \cite{davis2017three}. Under different regularity assumptions of the objective function, this approach can certify sublinear/linear convergence of the algorithm. In particular, we showed that our bounds are tight for the case of linear convergence.


\begin{appendices} 
\section{} \label{appendix}
Throughout the proofs the function classes of $f, g, h$ are parameterized by $f \in \mathcal{F}(m_f, L_f)$, $g \in \mathcal{F}(m_g, L_g)$, $h \in \mathcal{F}(m_h, L_h)$. We denote $Q_f = Q(m_f, L_f)$, $Q_g = Q(m_g, L_g)$, $Q_h = Q(m_h, L_h)$.
\subsection{Proof of Theorem~\ref{Thm2}}\label{pro:thm2}
\begin{proof}
In this theorem, we assume $m_f = m_g = m_h =0, L_h < L_f = L_g = \infty$.
		Define vector $v_k$ as
		\begin{equation}\label{eq:vector}
		v_k=
		\begin{bmatrix}
		(x_B^{k}-x_B^{\star})^{T} \ (y^k-y^{\star})^{T}  \ (x_A^{k}-x_A^{\star})^{T} \ (z^k-z^{\star})^{T}\\
		\end{bmatrix}^{T}.
		\end{equation}
		For the Lyapunov function $V_k$ in \eqref{eq:Lemma1Lyap}, it can be easily checked that 
		\begin{equation*}
		V_{k+1}-V_{k}=v_k^T W_0 v_k
		\end{equation*}
		since $x_A^\star = x_B^\star$. Noting that
		\begin{equation*}
		\begin{aligned}
		&v_k^{T} Q_1 v_k \\
		&= \begin{bmatrix}
		z^k-z^{\star}\\
		x_B^k-x_B^{\star}
		\end{bmatrix}^{T}
		\begin{bmatrix}
		0&\alpha I_d \\
		I_d&-I_d
		\end{bmatrix}^{T} Q_g 
		\begin{bmatrix}
		0&\alpha I_d \\
		I_d&-I_d
		\end{bmatrix}
		\begin{bmatrix}
		z^k-z^{\star}\\
		x_B^k-x_B^{\star}
		\end{bmatrix} \\
		& = \alpha^2 \begin{bmatrix}
		x_B^k - x_B^\star \\
		\frac{z^k - x_B^k}{\alpha} - \frac{z^\star - z_b^\star}{\alpha}
		\end{bmatrix}^{T} Q_g 
		\begin{bmatrix}
		x_B^k - x_B^\star \\
		\frac{z^k - x_B^k}{\alpha} - \frac{z^\star - z_b^\star}{\alpha}
		\end{bmatrix}\\
		& = \alpha^2 \begin{bmatrix}
		x_B^k - x_B^\star \\
		\partial g(x_B^k) - \partial g(x_B^\star)
		\end{bmatrix}^{T} Q_g 
		\begin{bmatrix}
		x_B^k - x_B^\star \\
		\partial g(x_B^k) - \partial g(x_B^\star)
		\end{bmatrix} \\
		& \geq 0
		\end{aligned}
		\end{equation*}
		for all $k$ where the third equality comes from~\eqref{eq:UpdatingRule} and the last inequality applies the property of the incremental quadratic constraints.
		
		Similarly, applying the alternations in~\eqref{eq:UpdatingRule} and incremental quadratic constraints on $h$ and $f$, we have that
		\begin{equation*}
		\begin{aligned}
		 & v_k^T Q_2 v_k \\
		& = \alpha^2 \begin{bmatrix}
		x_B^k-x_B^{\star}\\
		\nabla h(x_B^k) - \nabla h(x_B^\star)
		\end{bmatrix}^T Q_h \begin{bmatrix}
		x_B^k-x_B^{\star}\\
		\nabla h(x_B^k) - \nabla h(x_B^\star)
		\end{bmatrix} \\
		& \ge 0,
		\end{aligned}
		\end{equation*}
and
		\begin{equation*}
		\begin{aligned}
		&v_k^T Q_3 v_k\\
		& = \alpha^2 \begin{bmatrix}
		x_A^k - x_A^\star \\ \partial f(x_A^k) - \partial f(x_A^\star)
		\end{bmatrix}^T Q_f \begin{bmatrix}
		x_A^k - x_A^\star \\ \partial f(x_A^k) - \partial f(x_A^\star)
		\end{bmatrix} \geq 0 
		\end{aligned}
		\end{equation*}
for all $k$. If there exist $\sigma_1, \sigma_2, \sigma_3 \geq 0$ such that \eqref{eq:Case1LMI} holds, we obtain that
\begin{equation} \label{eq:matrixInqVk}
v_k^T W_0 v_k + \sigma_1 v_k^T Q_1 v_k +  \sigma_2 v_k^T Q_2 v_k + \sigma_3 v_k^T Q_3 v_k \le 0
\end{equation}
and the last three terms on the left-hand side of~\eqref{eq:matrixInqVk} are non-negative. As a result, we have $v_k^T W_0 v_k \le 0$ which leads to $V_{k+1} \le V_k$ for all $k$ and certifies the sublinear convergence~\eqref{eq:Thm2Lyap} of the TOS algorithm.
\end{proof}
		
\subsection{Proof of Theorem~\ref{thm:Case2}} \label{pro:Thm5}
\begin{proof}
In Theorem~\ref{thm:Case2} we assume $m_f = m_g = m_h = 0$ and $L_f, L_h < L_g = \infty$. From the fact that the function $f$ is convex and $L_f$-Lipschitz differentiable, we have that 
\begin{equation*}
		\begin{aligned}
		&f(x_A^k)-f(x_A^{\star}) \le {\nabla f(x_A^{k})}^T (x_A^k - x_A^{\star}),\\
		&f(x_B^k)-f(x_A^k) \le {\nabla f(x_A^{k})}^T (x_B^k - x_A^k) + \frac{L_f}{2} \lVert x_B^k - x_A^k \rVert^2.
		\end{aligned}
		\end{equation*}
Since $x_B^{\star}=x_A^{\star}$, adding up the above two inequalities, we have 
		\begin{equation} \label{eq:difference_f}
		f(x_B^{k})-f(x_B^{\star}) \le {\nabla f(x_A^{k})}^T (x_B^k - x_B^{\star}) + \frac{L_f}{2} \lVert x_B^k - x_A^k \rVert^2.
		\end{equation}
From the convexity of the function $g$, the following inequality 
\begin{equation} \label{eq:difference_g}
g(x_B^k)-g(x_B^{\star}) \le {\partial g(x_B^{k})}^T (x_B^k - x_B^{\star})
\end{equation} holds. Besides, since $h$ is $L_h$-Lipschitz differentiable, we have~\cite{lessard2016analysis}
\begin{equation} \label{eq:difference_h}
\begin{aligned}
h(x_B^k)-h(x_B^{\star}) \le & {\nabla h(x_B^k)}^T (x_B^k - x_B^{\star}) \\
		& - \frac{1}{2 L_h} \lVert \nabla h(x_B^k) - \nabla h(x_B^{\star}) \rVert^2.
		\end{aligned}
		\end{equation}
Adding up \eqref{eq:difference_f}, \eqref{eq:difference_g} and \eqref{eq:difference_h}, we have that 
		\begin{equation*}
		\begin{aligned}
		 & F(x_B^k)-F(x_B^{\star}) \\ & \le (\nabla f(x_A^k)+\partial g(x_B^k)+\nabla h(x_B^k))^T (x_B^k - x_B^{\star})\\
		& +\frac{L_f}{2} \lVert x_B^k - x_A^k \rVert^2
		- \frac{1}{2 L_h} \lVert \nabla h(x_B^k) - \nabla h(x_B^{\star}) \rVert^2.
		\end{aligned}
		\end{equation*}
It can be easily verified that $V_{k+1}-V_k \le v_k^{T} W_1 v_k$ for all $k$ if we define $v_k$ as 
		\begin{equation*}
		v_k = \begin{bmatrix}
		(x_B^{k}-x_B^{\star})^{T} \ (y^k-y^{\star})^{T} \ (x_A^{k}-x_A^{\star})^{T} \ (z^k-z^{\star})^{T}\\
		\end{bmatrix}^{T}.
		\end{equation*}
Using the same method as in Appendix~\ref{pro:thm2}, we conclude that \eqref{eq:Case2Convergence} holds for all $k$ if~\eqref{eq:Case2LMI} has a feasible solution.
\end{proof}
\subsection{Proof of Theorem~\ref{thm:LinearConvergence}}\label{pro:thm6}
\begin{proof}
Let $v_k$ be the same as \eqref{eq:vector}. Using the definition of $z^{k+1}$ in Algorithm~\ref{alg:TOS}, we can write
\begin{equation*}
V_{k+1}-\rho^2 V_k=v_k^T W_2 v_k.
\end{equation*} 
Then the same methods in the proof Appendix~\ref{pro:thm2} and \ref{pro:Thm5} apply here. If \eqref{eq:LinearConvergenceMI} holds, then $v_k^T W_2 v_k \le 0$, which means $V_{k+1} \le \rho^2 V_k$ and linear convergence~\eqref{eq:LinearConvergence} holds.
\end{proof}

\subsection{Duality} \label{append:duality}
In the linear convergence analysis of the TOS algorithm, to show the duality between our SDP formulation in Section~\ref{sec:linear} and the SDP in~\cite[Eq.(9)]{ryu2018operator}, we consider the following problem with the notation in Theorem~\ref{thm:LinearConvergence} for fixed stepsize $\alpha$ and relaxation parameter $\lambda$:
\begin{equation} \label{eq:equivalenceFormulation}
\begin{aligned}
\underset{\rho^2, \sigma_1, \sigma_2, \sigma_3}{\text{minimize}} & \ \rho^2 \\
\text{subject to} & \ G^T (W_O - \rho^2 W_I +\sigma_1 Q_1+\sigma_2 Q_2+\sigma_3 Q_3)G \preceq 0\\
& \quad \sigma_1,\sigma_2,\sigma_3 \geq 0,
\end{aligned}
\end{equation}  
\normalsize
where 
\begin{equation*}
\begin{aligned}
& W_O = 
\begin{bmatrix}
\lambda^2&0&-\lambda^2&-\lambda\\
0&0&0&0\\
-\lambda^2&0&\lambda^2&\lambda\\
-\lambda&0&\lambda&1
\end{bmatrix} \otimes I_d, \\
& W_I = \begin{bmatrix}
0&0&0&0\\
0&0&0&0\\
0&0&0&0\\
0&0&0&1
\end{bmatrix} \otimes I_d, \ 
G = \begin{bmatrix}
0 & 0 & 1 & 0 \\ 
-1 & 0 & 2 & -1\\
0 & 1 & 0 & 0 \\
1 & 0 & 0 & 0
\end{bmatrix} \otimes I_d.
\end{aligned}
\end{equation*}
The matrix inequality in~\eqref{eq:equivalenceFormulation} is equivalent to~\eqref{eq:LinearConvergenceMI} since $W_2 = W_O - \rho^2 W_I$ and $G$ is invertible. It is not hard to show that the Lagrangian dual of~\eqref{eq:equivalenceFormulation} is
\begin{equation}
\begin{aligned}
\underset{Z}{\text{maximize}} & \quad  \Tr(G^T W_O G Z) \\
\text{subject to} & \quad \Tr(G^T Q_1 G Z) \succeq 0 \\
& \quad \Tr(G^T Q_2 G Z) \succeq 0\\
& \quad \Tr(G^T Q_3 G Z) \succeq 0 \\
& \quad \Tr(G^T W_I G Z) = 1 \\
& \quad Z \succeq 0,
\end{aligned}
\end{equation}
which is equivalent to ~\cite[Eq.(9)]{ryu2018operator} under Assumption~\ref{linearassume}. Following the strong duality proof in~\cite{ryu2018operator}, we can show that our Lyapunov-function-based SDP~\eqref{eq:equivalenceFormulation} is the dual of that in~\cite{ryu2018operator} and hence achieves the same tight bounds on $\rho^2$.
\end{appendices}

	
	\bibliographystyle{ieeetr}
	\bibliography{reference}

\end{document}